\documentclass[12pt]{article}
\def\equ{\not\equiv }
\usepackage{amsthm}
\usepackage{amsfonts}
\theoremstyle{definition}
\newtheorem{theorem}{Theorem}
\newtheorem{definition}{Definition}
\newtheorem{lemma}{Lemma}

\newcommand{\sss}[1]{\sum\limits_{#1\in S}}
\newcommand{\intst}{\int_s^t}
\newcommand{\intsss}{\int_s^t \sum\limits_{k\in S}}
\newcommand{\ol}{\overline}

\newcommand{\e}[3]{\  e^{\int_{#1}^{#2} q_{ii}(#3)d#3}\ }

\newcommand{\be}{\begin{equation}}
\newcommand{\ee}{\end{equation}}
\newcommand{\bee}{\begin{eqnarray*}}
\newcommand{\eee}{\end{eqnarray*}}
\newcommand{\beb}{\begin{eqnarray}}
\newcommand{\eeb}{\end{eqnarray}}

\def\footnotemark{}

\title{Existence and Regularity of a Nonhomogeneous
Transition Matrix under Measurability
Conditions$^*$\thanks{$^*$Supported by NSFC and RFDP}}
\author{Liuer Ye$^1$, Xianping Guo$^1$, and On$\acute{\rm{e}}$simo
Hern$\acute{\rm{a}}$ndez-Lerma$^{2,\S}$\footnote{$^\S\!\!$ The research of this author was partially supported by CONACYT grant 45693-F} \\
$^1$ The School of Mathematics and Computational Science \\
 \ Zhongshan University, Guangzhou, P. R. China\\
  $^2$Departamento de
Matem\'aticas,  CINVESTAV--IPN,\\
  A.Postal 14-740, M$\acute{\rm{e}}$xico D.F.  07000, Mexico\\
E--mail: (LY) yeliuer@hotmail.com, (XPG) mcsgxp@mail.sysu.edu.cn, \\
(OHL) ohernand@math.cinvestav.mx\\
\\
--- January 2007---\\
---Revised: November 2007---}
\date{}
\begin{document}
\maketitle

{\bf Abstract:}  This paper is about the existence and regularity of
the transition probability matrix of a nonhomogeneous
continuous-time Markov process with a countable state space. A
standard approach to prove the existence of such a transition matrix
is to begin with a {\it continuous} (in $t$) and {\it conservative}
matrix $Q(t)=[q_{ij}(t)]$ of nonhomogeneous transition rates
$q_{ij}(t)$, and use it to construct the transition probability
matrix. Here we obtain the same result except that the $q_{ij}(t)$
are only required to satisfy a mild {\it measurability} condition,
and $Q(t)$ may {\it not} be conservative. Moreover, the resulting
transition matrix is shown to be the {\it minimum} transition matrix
and, in addition, a necessary and sufficient condition for it to be
{\it regular} is obtained. These results are crucial in some
applications of nonhomogeneous continuous-time Markov processes,
such as stochastic optimal control problems and stochastic games,
which motivated this work in the first place.


{\bf AMS 2000} subject classifications: 60J27, 60J35, 60J75

{\bf Suggested running head}: Nonhomogeneous transition matrices
under measurability conditions


\section{INTRODUCTION} \label{intro}
Nonhomogeneous continuous-time Markov processes have applications in
a wide variety of contexts, including stochastic control problems
\cite{GH03b,Guo03,Guo02,HuQY96,K.P75, Puterman94,WuCB97a},
stochastic games \cite{GH05, GH03a,Sennott99,Sinha04}, and queueing
systems and stochastic networks \cite{B03,Sennott99,Stidham93}, to
name a few. As is well known, such a Markov process is uniquely
determined by its transition probability matrix
$P(s,t)=[P_{ij}(s,t)]$ (for $i, j \in S$ and $0\le s \le t$), which
in turn is usually constructed from a given matrix
$Q(t)=[q_{ij}(t)]$ of transition rates $q_{ij}(t)$, for $t\geq 0$.
Therefore, a natural question is, under what conditions on $Q(t)$ is
the transition matrix $P(s,t)$ uniquely determined?

To answer this question, a standard approach---which can be traced
back to Feller's 1940 paper \cite{F40}---is to assume that the
transition rates $q_{ij}(t)$ are {\it continuous} in $t\geq 0$;
see, e.g.,
\cite{B03,GS96,GH05,GH03a,GH03b,Guo03,HuQY96,LiuJY04,WuCB97a}.
This continuity requirement, however, imposes severe restrictions
in some applications, for instance, in stochastic control and game
theory, where discontinuous ``policies'' typically lead to
discontinuous transition rates. To illustrate this situation, which was
the main motivation for this work, let us consider the following
example.

{\bf Example.}
Consider a single-server queueing system in which the state
variable $i$ denotes the number of jobs in the system at each time $t\ge 0$.
Suppose that a controller wants to control the system's
service rate $\mu$ according to the current state $i\in
S:=\{0,1,\ldots\}$. When the state is $i$ at some time $t$,
the controller takes a service rate $\mu$ from a given finite set $A(i)$
of available ``actions'' at state $i$; then the system transfers to
another state $j$ according to the transition rate induced by the chosen
$\mu$, and the
process is repeated. The choice of service rates is done according to
so-called {\it control policies} $\pi=\{\pi_t,\;t\geq 0\}$. If the controller
is using a particular policy $\pi=\{\pi_t,\;t\geq 0\}$ and the state at time
$t$ is
$i\in S$, then the controller takes the service rate $\pi_t(i)\in A(i)$.
To be more specific, consider the policy $\pi$ given by
\begin{equation}
\pi_t(i):= \sum_{k=0}^{\infty} \mu_k(i) {\bf 1} _{[k, k+1)} (t),
\end{equation}
with $\mu_k(i) \in A(i)$ for all $i\in S$ and $k\ge 0$. Hence, when
using this policy, if the present state is $i$, then the controller
chooses the action $\mu_k(i)\in A(i)$ during the time interval
$[k,k+1)$. Obviously, $\pi_t(i)$ is measurable in $t\ge 0$, but
not continuous. Similarly, if $q^\mu_{ij}$ denotes the transition rate
from $i$ to $j$ under $\mu\in A(i)$, then the matrix $Q^\pi(t)=[q^\pi_{ij}(t)]$ of transition rates when using $\pi$ has elements
$$
q^\pi_{ij}(t):=\sum^\infty_{k=0} q^{\mu_k(i)}_{ij}{\bf 1}_{[k,k+1)}(t).
$$
Therefore, the transition rates are measurable in $t\geq 0$ but {\it not}
continuous, and so we cannot use the standard Markov chain theory to show the
existence of transition probability functions $P^\pi_{ij}(s,t)$ induced by
the discontinuous control policy $\pi$ defined in (1.1). An analogous
situation occurs in stochastic game problems \cite{GH05,GH03a}, where
discontinuous ``strategies'' usually lead to
discontinuous transition rates.  This is precisely what motivated our
paper--- to establish the existence and regularity
of a nonhomogeneous transition matrix without requiring the transition
rates to be continuous.

Summarizing, our main objective is to replace the above-mentioned
continuity requirement by a mild measurability condition under
which we obtain the existence and uniqueness of a transition
matrix $P(s,t)$, even if the $Q(t)$ matrix is {\it not}
conservative. (See Section 2 for definitions.) In fact, the
existence and uniqueness of $P(s,t)$ under a measurability
condition have been considered in \cite{LiuJY04}, but the results
there are stated {\it without} proofs and assuming,
in addition, that $Q(t)$ is {\it conservative}. We also
obtain some new results (see, for instance, Theorems 2(i) and 2(iii),
Theorem 3(ii)).

In this paper, firstly, we introduce a precise definition of a
nonhomogeneous transition matrix (Definition 1), which is weaker
than previous definitions, e.g., as in \cite{B03,GS96,HuDH83}. Even
in this weaker context, we can obtain some key properties of the
transition matrix (Theorem 1). Secondly, given a $Q(t)$ matrix
satisfying our measurability condition, we construct a
nonhomogeneous transition matrix (Theorem 2), and, finally, we give
a necessary and sufficient condition for this transition matrix to
be unique and regular (Theorem 3).

The rest of this paper is organized as follows. In Section 2 we
present the definitions and main results concerning nonhomogeneous
transition matrices. The proofs of our results are all given in
Section 3. In Section 4 we state some conclusions.


\section{MAIN RESULTS}
\label{sec1} \setcounter{equation}{0}

\subsection{Nonhomogeneous Pretransition Matrices} \label{sec1:subsec1}

Throughout the following $S$ denotes a given countable set.

\begin{definition}
\label{def}
 A real-valued matrix
$P(s,t)=(P_{ij}(s,t),\ i,j\in S)$, defined for all
$0\le s\le t<\infty$, is called a nonhomogeneous {\it pretransition matrix} if it satisfies the following for every $i,j\in S$ and $0\leq s\leq t<\infty$:
\begin{enumerate}
\item \be
 \label{s11}
 \!\!\!\!\!\!\!\!\!\!\!\!\!\!   P_{ij}(s,t)\ge 0,\  {\rm and} \  \sss{j} P_{ij}(s,t)\le 1;
\ee
\item \be
  \label{s12}
 \!\!\!\!\!\!\!\!\!\!\!\!\!\!   P_{ij}(s,t)=\sss{k}
P_{ik}(s,u)P_{kj}(u,t) \ \ \   \forall  s\le u\le t; \ee
\item \be
 \label{s13}
 \!\!\!\!\!\!\!\!\!\!\!\!\!\!  \lim_{h \to 0^+} | P_{ij}(s,s+h) - \delta_{ij} | =0\ \
 {\rm uniformly\  in}\  j\in S, \ {\rm and} \ P_{ij}(s,s)=\delta_{ij},
 \ee
where $\delta_{ij}$ stands for the Kronecker symbol $(\delta_{ij}=0
 $ if $  i\ne j\ ;  \delta_{ij}=1  $ if $  i=j).$
 \end{enumerate}
If in addition $ \sum_{j\in S} P_{ij}(s,t)\equiv 1$ for every $i\in
S$,  then $P(s,t)$ is said to be a nonhomogeneous {\it transition
probability matrix}, and its elements $P_{ij}(s,t)$ are called
transition functions.

\end{definition}

The equation (2.2) is known as the {\it Chapman-Kolmogorov}
$(C\!\!-\!\!K)$ {\it equation}. Our definition of a pretransition
matrix is weaker than that in \cite{B03,GS96,HuDH83}, but still we
can obtain most of the standard results. In particular, the
following theorem is essentially the same as Theorem 1.1 and Theorem
1.2 in Chapter 3 of \cite{HuDH83}.

\begin{theorem}
\label{dd0} For any nonhomogeneous pretransition probability matrix
$P(s,t)$ we  have:
\begin{enumerate}
     \item \be \label{s14}
  \!\!\!\!\!\!\!\!\!\!\!\!\!\! P_{ii}(s,t)>0 \ \   \forall \  i\in S,\ 0\le s\le t<\infty;
 \ee
   \item \be
\label{s15}
 \!\!\!\!\!\!\!\!\!\!\!\!\!\!  |P_{ij}(u,t)-P_{ij}(v,t)| \le 1-P_{ii}(u\wedge v,u\vee v) \ \  \forall \  i,j\in S,  0\le u,v\le t,
\ee where $u\wedge v=\min(u,v), \ u\vee v=\max(u,v)$;
     \item $P_{ij}(s,t)$ is continuous in $s \in [0,t]$
(right-continuous in $0$, left-continuous in $t$), and uniformly in
$t\ge 0$ and $j\in S$;
     \item   $P_{ij}(s,t)$ is continuous in $t \in
[s,+\infty) $ (right-continuous in $s$), and uniformly in $j\in S$;
    \item  The following holds:
$$\lim_{t\to s^+} \frac{P_{ii}(s,t)-1}{t-s} = \lim_{t-s\to 0^+}
 \frac{P_{ii}(s,t)-1}{t-s} \ =:q_{ii}(s)\leq 0 ,$$
 $$\lim_{t\to s^+} \frac{P_{ij}(s,t)}{t-s} = \lim_{t-s\to 0^+}
 \frac{P_{ij}(s,t)}{t-s} \  =:q_{ij}(s)\geq 0 \ \ {\rm for} \ \  j\not= i;$$
    \item   For every $i\in S$ and $s\geq 0$, $\sum_{j\ne i}q_{ij}(s)\leq
-q_{ii}(s) $. (Hence, $q_{ij}(s)\leq q_i(s)$ for all $i,j\in S$ and
$s\geq 0$, where $q_i(s):=-q_{ii}(s)\geq 0$.)
\end{enumerate}
\end{theorem}
\begin{proof}
See subsection \ref{sec2:subsec1}.
\end{proof}

\subsection{Nonhomogeneous $Q(t)$-Transition Matrices}
\label{sec1:subsec2}

In this subsection we introduce the definition of a nonhomogeneous
$Q(t)$-{\it matrix} and a nonhomogeneous transition matrix induced
by $Q(t)$, which are related by our key hypothesis, Assumption A.

\begin{definition}
\label{def2}
For each $i,j\in S$, let $q_{ij}(t)$ be a
real-valued function defined on $[0,+\infty)$. The matrix function
$Q(t)=(q_{ij}(t), i,j\in S)$ is said to be a {\it nonhomogeneous} $Q(t)$-{\it matrix} on $S$ if for every $i,j\in S$ and $t\geq 0$ satisfies that:
\begin{enumerate}
\item  \be
 \label{s21}
  \!\!\!\!\!\!\!\!\!\!\!\!\!\!  0\le q_{ij}(t)<\infty \ \ {\rm if} \ \ i\neq j, \ \ {\rm and} \  \ 0\le -q_{ii}(t)<\infty;
 \ee
\item   \be
 \label{s22}
  \!\!\!\!\!\!\!\!\!\!\!\!\!\!  \sss{j} q_{ij}(t) \le 0 \ \ \forall \ i\in S.
\ee
\end{enumerate}
If in addition  $\sss{j} q_{ij}(t)=0 $ for all
$i\in S$ and $t \ge 0$, then we say that  $ Q(t)$ is {\it
conservative}.
\end{definition}

We now introduce our measurability-integrability assumption with respect to the Lebesgue measure on $[0,+\infty)$.

\noindent {\bf Assumption A} \  Let $Q(t)$ be a given nonhomogeneous
$Q(t)$-matrix. For every $b\geq a\geq 0$ and $i,j\in S$, we have:
$q_{ij}(t)$ is Borel-measurable in $t\in[a,b]$ if $i\neq j$, and
$q_{ii}(t)$ is integrable on $[a,b]$.

The following definition relates a $Q(t)$-matrix and a pretransition
matrix. (The abbreviations {\it a.e.} (almost everywhere) and {\it
a.a.} (almost all) refer to the Lebesgue measure.)

\begin{definition}
Let $Q(t)=(q_{ij}(t),  i,j\in S)\ (t\ge 0)$ be a nonhomogeneous
$Q(t)$-matrix satisfying Assumption A. If a nonhomogeneous
pretransition  matrix $P(s,t)=(P_{ij}(s,t),  i,j\in S)\ (0\le s\le
t<\infty)$ satisfies that, for every $i,j\in S$ and a.a. $s\geq 0$, the partial
derivatives
\begin{enumerate}
\item   $ \displaystyle  \frac{\partial P_{ij}(s,t)}{\partial s}$ and
\item  \be
 \label{s23}
 \!\!\!\!\!\!\!\!\!\!\!\!\!\! \frac{\partial}{\partial t} P_{ij}(s,t)|_{t=s+}=\lim_{h\to 0^+}
\frac{P_{ij}(s,s+h)-\delta_{ij}}{h} = q_{ij}(s)
    \ee
 \end{enumerate}
exist, then
$P(s,t)$ is called a {\it nonhomogeneous} $Q(t)$-{\it transition
matrix}. Further, such a $Q(t)$-transition matrix is said to be {\it
regular} if $P(s,t)$ is a transition probability matrix and it is
the unique transition probability matrix that satisfies (\ref{s23}).
We call $P_{ij}(s,t)$ a nonhomogeneous $Q(t)$-{\it transition
function} or simply a {\it $Q(t)$-function}.
\end{definition}

\subsection{ Existence of a Nonhomogeneous $Q(t)$-Transition Matrices}
\label{sec1:subsec3}

The question now is, given a nonhomogeneous $Q(t)$-matrix, how can
we ensure the existence of a nonhomogeneous $Q(t)$-transition
matrix? The following Theorem 2 shows that this can be done if the
$Q(t)$-matrix satisfies Assumption A. To obtain a {\it regular}
$Q(t)$-transition matrix, however, we have to go a bit further and
construct a {\it minimum} $Q(t)$-transition matrix on which we can
readily impose conditions for it to be regular. This is done in
Theorem 3.

\begin{theorem}
 \label{dd1}
Given a nonhomogeneous $Q(t)$-matrix satisfying Assumption A, let
\be
    \label{s31}
   \quad\quad\quad\quad\quad\quad\quad\quad  d_{ij}(u) := \delta_{ij}(-q_{ii}(u))\ \ \forall \ i,j\in S,\  u\ge 0,
\ee
 and for each $i,j\in S$ and $0\leq s\leq t<\infty$, define
recursively
 \be \label{s32}
  \quad\quad\quad\quad\quad\quad\quad\quad\quad\quad  \ol P_{ij}^{(0)}(s,t) := \delta_{ij}\ \ e^{\intst q_{ii}(u)du},
\ee
 \be \label{s33}
     \ol P_{ij}^{(n+1)}(s,t)  := \int_s^t \sum\limits_{k\in S}\  e^{\int_s^u
     q_{ii}(v)dv} [q_{ik}(u)+d_{ik}(u)] \ol P_{kj}^{(n)}(u,t) du \ \ \forall n\ge 0.
\ee
 Let
 \be \label{s34}
   \quad\quad\quad\quad\quad\quad\quad\quad\quad\   \ol P_{ij}(s,t)  := \sum_{n=0}^{\infty}\ \  \ol P_{ij}^{(n)}(s,t).
\ee
 Similarly, for each $i,j\in S$ and $0\leq s\leq t<\infty$,
define recursively
 \be
\label{q01}
   \quad\quad\quad\quad\quad\quad\quad\quad\quad  \ol Q_{ij}^{(0)}(s,t) := \delta_{ij}\ \ e^{\intst q_{ii}(u)du},
 \ee
 \be
 \label{q02}
     \ol Q_{ij}^{(n+1)}(s,t)  := \int_s^t \sum\limits_{k\in S}\ \ol
     Q_{ik}^{(n)}(s,u)\
     e^{\int_u^t  q_{jj}(v)dv} [q_{kj}(u)+d_{kj}(u)] du \ \ \forall n\ge 0.
\ee
 Let
 \be
  \label{q03}
     \quad\quad\quad\quad\quad\quad\quad\quad\quad\  \ol Q_{ij}(s,t)  := \sum_{n=0}^{\infty}\ \  \ol Q_{ij}^{(n)}(s,t).
\ee
 Then, for every $i,j\in S$ and $0\leq s\leq t<\infty$, we have
 \vspace{2mm}
\begin{enumerate}
\item  \be
 \label{pqe}
  \!\!\!\!\!\!\!\!\!\!\!\!\!\!  \ol P_{ij}^{(n)}(s,t) = \ol Q_{ij}^{(n)}(s,t) \ {\rm for\  all} \ n\ge
 0, \ {\rm and \ so} \  \ol P_{ij}(s,t)=\ol Q_{ij}(s,t); \ee
 \item \be
 \label{s35}
 \!\!\!\!\!\!\!\!\!\!\!\!\!\!  \ol P_{ij}(s,t) = \int_s^t \sum\limits_{k\in S} \ol P_{ik}(s,v)q_{kj}(v) dv +\delta_{ij};
\ee
\item
 \be
 \label{qbackw}
 \!\!\!\!\!\!\!\!\!\!\!\!\!\!  \ol P_{ij}(s,t) = \int_s^t \sum\limits_{k\in S} q_{ik}(v) \ol P_{kj}(v,t) dv +\delta_{ij};
\ee
 \item  $ \ol P(s,t) =(\ol P_{ij}(s,t), i,j\in S) $ is a
nonhomogeneous $Q(t)$-transition matrix.
\end{enumerate}
\end{theorem}
\begin{proof}
See subsection \ref{sec2:subsec2}.
\end{proof}

The following theorem states our main results in this paper. It
shows that, whether $Q(t)$ is conservative or not, the
$Q(t)$-transition matrix $\ol P(s,t)$ constructed in Theorem 2 is
{\it minimum}; see (\ref{s55}). Moreover, it gives a reasonably mild
necessary and sufficient condition for $\ol P(s,t)$ to be regular;
see (\ref{s001}).

\begin{theorem}
 \label{dd3}
Assume that $Q(t) $ is a nonhomogeneous $Q(t)$-matrix satisfying
Assumption A, and let $\ol P(s,t)$ be the nonhomogeneous
$Q(t)$-transition matrix in Theorem 2. Then:
 \begin{enumerate}
  \item  $ \ol{P}(s,t)$
is the minimum $Q(t)$-transition matrix, that is, for any
nonhomogeneous $Q(t)$-transition matrix $P(s,t)=(P_{ij}(s,t),\
i,j\in S) $, we have
 \be
  \label{s55}
 \quad\quad\quad\quad   P_{ij}(s,t) \ge \ol P_{ij}(s,t) \ \ \forall \ i,j\in S,  \ 0\le s\le t<\infty.
\ee
  \item  $ \ol P(s,t)$ is regular if and only if
 \be
  \label{s001}
   \quad\quad\quad   \sss{j} \int_s^t \sum\limits_{k\in S} \ol P_{ik}(s,v)q_{kj}(v) dv \equiv
    0 \ \ \forall \ i\in S, \ 0\le s\le t<\infty.
 \ee
 \end{enumerate}
\end{theorem}
\begin{proof}
See subsection \ref{sec2:subsec4}.
\end{proof}


\section{PROOFS}
\label{sec2} \setcounter{equation}{0}

\subsection{Proof of Theorem \ref{dd0}}
\label{sec2:subsec1}

To prove Theorem \ref{dd0} we will use the following lemma in
which, for $h>0$ and an integer $m\geq 1$, we denote by
${}_AP_{i,B}(s,s+mh)$ the probability of transition from the state
$i$ at time $s$ to the set $B$ at time $s+mh$ while avoiding the set
$A$ at times $s+kh$ for $k=1,\ldots,m-1$. Observe that
$${}_AP_{i,B}(s,s+h)=P_{i,B}(s,s+h),$$ and, for $m\geq 2$,
$${}_AP_{i,B}(s,s+mh)=\sum_{r\not\in A} {}_AP_{ir}(s,s+(m-1)h) P_{r,B}(s+(m-1)h,s+mh).$$

\begin{lemma}
     \label{le}
 Let $P(s,t)=(P_{ij}(s,t),  i,j\in S)$ be a nonhomogeneous pretransition matrix and, for $0\leq s<t$, let $0<h<t-s, \ n:=[h^{-1}(t-s)]$, and $1\leq m\leq n$. Then:
\begin{enumerate}
\item
  For every $i\in S$ and  $B\subset S$,
  \beb
   \label{s004}
  \quad\quad    P_{i,B}(s,t) & =& \sum_{l=1}^m \sum_{k\in A}
    {}_AP_{ik}(s,s+lh) P_{k,B}(s+lh,t) \nonumber \\
    \label{s004}
    & &   + \sum_{k\not\in A} {}_AP_{ik}(s,s+mh) P_{k,B}(s+mh,t).
  \eeb
 \item  For every $0<\varepsilon<1/3$, there exists $0<\delta<1$ such that
   when $t-s<\delta$, we have
  \be
   \label{s002}
  \quad\quad P_{ij}(s,t)\ge (1-3\varepsilon) \sum_{l=1}^n  P_{ij}(s+(l-1)h,s+lh)\ \ \forall \ j\neq i.
   \ee
  \end{enumerate}
 \end{lemma}

\begin{proof}
(i) We will use induction on $m$. For $m=1$, (\ref{s004}) holds by
the definition of ${}_AP_{i,B}(s,s+mh)$ and the C-K equation
(\ref{s12}). Now assume that (\ref{s004}) holds for $m-1$. We will
prove that it holds for $m$. Indeed, by the definition of
${}_AP_{i,B}(s,s+mh)$, the C-K equation (\ref{s12}) again and the
induction hypothesis, we have
 \bee
    \lefteqn{ P_{i,B}(s,t)} \\
    & =& \sum_{l=1}^{m-1} \sum_{k\in A}
    {}_AP_{ik}(s,s+lh) P_{k,B}(s+lh,t) \\
    & &  + \sum_{k\not\in A} {}_AP_{ik}(s,s+(m-1)h) P_{k,B}(s+(m-1)h,t)\\
     &= & \sum_{l=1}^m  \sum_{k\in A} {}_AP_{ik}(s,s+lh) P_{k,B}(s+lh,t)
      -  \sum_{k\in A}  {}_AP_{ik}(s,s+mh) P_{k,B}(s+mh,t) \\
     & & + \sum_{k\not\in A} {}_AP_{ik}(s,s+(m-1)h) P_{k,B}(s+(m-1)h,t) \\
    & &  + \sum_{k\not\in A} {}_AP_{ik}(s,s+mh) P_{k,B}(s+mh,t)
     -   \sum_{k\not\in A} {}_AP_{ik}(s,s+mh) P_{k,B}(s+mh,t)\\
    &= & \sum_{l=1}^m \sum_{k\in A} {}_AP_{ik}(s,s+lh) P_{k,B}(s+lh,t)
     + \sum_{k\not\in A} {}_AP_{ik}(s,s+mh)  P_{k,B}(s+mh,t) \\
    & & + \sum_{k\not\in A} {}_AP_{ik}(s,s+(m-1)h) P_{k,B}(s+(m-1)h,t)\\
    & & - \sum_{k\in S} {}_AP_{ik}(s,s+mh) P_{k,B}(s+mh,t)\\
    &= & \sum_{l=1}^m \sum_{k\in A} {}_AP_{ik}(s,s+lh) P_{k,B}(s+lh,t)
      + \sum_{k\not\in A} {}_AP_{ik}(s,s+mh) P_{k,B}(s+mh,t).
   \eee
This completes the induction.

 \vspace{3mm}
 (ii) Taking $B=A=\{j\}$ and $m=n$ in (\ref{s004}), we obtain
  \be
  \label{s005}
   P_{ij}(s,t)=\sum_{l=1}^n  {}_jP_{ij}(s,s+lh)
     P_{jj}(s+lh,t)+\sum_{k\ne j} {}_jP_{ik}(s,s+nh)
  P_{kj}(s+nh,t).
  \ee
On the other hand, (\ref{s13}) implies that for any given
$0<\varepsilon<1/3$ and $j\ne i$, there exists $0<\delta<1$ such
that when $0<t-s<\delta$, we have
  $ P_{ii}(s,t)>1-\varepsilon,\  P_{jj}(s,t)>1-\varepsilon,\ {\rm and} \
  P_{ij}(s,t)<\varepsilon $. These facts together with
(\ref{s005}) imply that for  $ h<t-s<\delta$ and $j\neq i$
    \bee
   \quad\quad\quad\quad\quad   \varepsilon > 1-P_{ii}(s,t) &\ge& \sum_{k\ne i} P_{ik}(s,t) \ge P_{ij}(s,t) \\
      &\ge& \sum_{l=1}^n {}_jP_{ij}(s,s+lh) P_{jj}(s+lh,t) \\
      &\ge& (1-\varepsilon) \sum_{l=1}^n {}_jP_{ij}(s,s+lh),
      \eee
so
 $$ \sum_{l=1}^n {}_jP_{ij}(s,s+lh) \le \frac{\varepsilon}{1-\varepsilon}.$$
Note that, for each $1\le l\le n$,
$$P_{ii}(s,s+lh) =  {}_jP_{ii}(s,s+lh)
+ \sum_{m=1}^{l-1} {}_jP_{ij}(s,s+mh) P_{ji}(s+mh,s+lh). $$
Thus
 \bee
 \quad \quad\ {}_jP_{ii}(s,s+lh)&=& P_{ii}(s,s+lh)-\sum_{m=1}^{l-1} {}_jP_{ij}(s,s+mh)
  P_{ji}(s+mh,s+lh)\\
  &\ge& P_{ii}(s,s+lh)-\sum_{m=1}^{l-1} {}_jP_{ij}(s,s+mh)\\
  &\ge& 1-\varepsilon-\frac{\varepsilon}{1-\varepsilon},
  \eee
and so
 \bee
 \quad  P_{ij}(s,t) &\ge& \sum_{l=1}^n {}_jP_{ij}(s,s+lh)
   P_{jj}(s+lh,t)\\
   &=& \sum_{l=1}^n  \sum_{r\ne j} {}_jP_{ir}(s,s+(l-1)h) P_{rj}(s+(l-1)h,s+lh)
   P_{jj}(s+lh,t)\\
   &\ge& \sum_{l=1}^n   {}_jP_{ii}(s,s+(l-1)h) P_{ij}(s+(l-1)h,s+lh)
   P_{jj}(s+lh,t)\\
   &\ge& (1-\varepsilon-\frac{\varepsilon}{1-\varepsilon})
    \sum_{l=1}^n P_{ij}(s+(l-1)h,s+lh) (1-\varepsilon)\\
    &\ge& (1-3\varepsilon) \sum_{l=1}^n P_{ij}(s+(l-1)h,s+lh).
    \eee
This completes the verification of (3.2) and also the proof of the
lemma.
\end{proof}

With Lemma \ref{le}, we can easily prove Theorem \ref{dd0}.


\noindent
{\it Proof of Theorem \ref{dd0}.}

(i) By the C-K equation (\ref{s12}), we obtain
$$P_{ii}(s,t)\ge \prod_{k=1}^n P_{ii}\left(s+\frac{k-1}{n} (t-s),s+\frac{k}{n} (t-s)\right)\ \ \forall n\ge 1.$$
Hence, for $n$ sufficiently large, (\ref{s14}) follows from (\ref{s13}).

(ii) Let $0\le u\le v\le t<\infty$. By the C-K equation
(\ref{s12})
$$ P_{ij}(u,t)-P_{ij}(v,t)=\sum_{k\ne i}
P_{ik}(u,v)P_{kj}(v,t)+(P_{ii}(u,v)-1)P_{ij}(v,t). $$
 Applying (\ref{s11}), we obtain
$$ P_{ij}(u,t)-P_{ij}(v,t)\ge (P_{ii}(u,v)-1)P_{ij}(v,t)\ge P_{ii}(u,v)-1, $$
$$ P_{ij}(u,t)-P_{ij}(v,t)\le \sum_{k\ne i}
P_{ik}(u,v)P_{kj}(v,t)\le \sum_{k\ne i} P_{ik}(u,v) \le
1-P_{ii}(u,v).$$
These inequalities yield (\ref{s15}).

(iii) Using (\ref{s15}) and (\ref{s13}), for each $h\ge0 $ we
have
$$ |P_{ij}(s+h,t)-P_{ij}(s,t)|\le 1-P_{ii}(s,s+h) \to 0 \ \ {\rm as} \ \ h \to 0^+,$$
$$ |P_{ij}(s,t)-P_{ij}(s-h,t)|\le 1-P_{ii}(s-h,s) \to 0 \ \ {\rm as} \ \ h \to 0^+,$$
which together with (\ref{s13}) yield (iii).

(iv) By the C-K equation (\ref{s12}), for each $ h\ge0 $ we
have
$$P_{ij}(s,t+h)-P_{ij}(s,t) = \sum_{k\ne j} P_{ik}(s,t)P_{kj}(t,t+h)- (1-P_{jj}(t,t+h)) P_{ij}(s,t).$$
Therefore
$$P_{ij}(s,t+h)-P_{ij}(s,t) \ge - \sss{k} (1-P_{kk}(t,t+h)) P_{ik}(s,t),$$
and
$$P_{ij}(s,t+h)-P_{ij}(s,t) \le  \sum_{k\ne j}
P_{ik}(s,t)P_{kj}(t,t+h)
  \le \sss{k} P_{ik}(s,t)(1-P_{kk}(t,t+h)).$$
Hence, as $h\rightarrow 0^+$,
$$ |P_{ij}(s,t+h)-P_{ij}(s,t)| \le \sss{k} P_{ik}(s,t)(1-P_{kk}(t,t+h))\rightarrow 0, $$
by (\ref{s13}) and the Dominated Convergence Theorem \cite{Ash00}.
It follows that $P_{ij}(s,t)$ is right-continuous in $t \in [s,+\infty)$,
uniformly in $j\in S$.

Similar arguments show that $P_{ij}(s,t)$ is left-continuous in $t \in [s,+\infty)$, uniformly in $j\in S$.

(v) To avoid trivial situations we suppose that $i\in S$ is not an
absorbing state, i.e., $P_{ii}(s,t)\equ 1$. For $0\leq s\leq
t<+\infty$, let $f(s,t):=-\log P_{ii}(s,t)$, which is a well-defined
function, nonnegative and finite. Since $P_{ii}(s,t)\ge
P_{ii}(s,u)P_{ii}(u,t),$ for $0\le s\le u \le t<\infty$,  we have
$f(s,t)\le f(s,u)+f(u,t)$. Now for each $s\ge 0$, let $q_i(s):=
\sup\limits_{s<t<\infty} \frac{f(s,t)}{t-s}$. We will next prove
that the limit of $\frac{f(s,t)}{t-s}$ exists and equals $q_i(s)$.

Obviously, by definition of $q_i(s)$,
$$ \limsup_{t-s \to 0^+} \frac{f(s,t)}{t-s} \le q_i(s).$$
Therefore, it is sufficient to argue that $ \liminf\limits_{t-s \to 0^+}
\frac{f(s,t)}{t-s} \ge q_i(s)$.\ \ Given any $0<h<t-s$, take $n$
such that $t-s=nh+\varepsilon,$ with $0\le \varepsilon<h$. Then
$$\frac{f(s,t)}{t-s}\le \frac{nh}{t-s}\frac{f(s,s+nh)}{nh}
+\frac{f(s+nh,t)}{t-s}.$$
Now take the limit of both sides as $h\to 0^+$, and note that $nh\to t-s\ {\rm as}
\ \varepsilon\to 0^+$, so that the continuity of $P_{ij}(s,t)$ implies that
$f(s+nh,t)=f(t-\varepsilon,t)\to 0$. Hence we have
  $$ \frac{f(s,t)}{t-s}\le \liminf_{h\to 0^+} \frac{f(s,s+nh)}{nh}
  = \liminf_{t-s \to 0^+} \frac{f(s,t)}{t-s},$$
and so
$$ \lim_{t-s \to 0^+} \frac{f(s,t)}{t-s} = q_i(s). $$
Finally, recalling the definition of $f(s,t)$, we have
  $$  \lim_{t-s\to 0^+} \frac{1-P_{ii}(s,t)}{t-s}
  = \lim_{t-s\to 0^+} \frac{1-e^{-f(s,t)}}{t-s}
  = \lim_{t-s\to 0^+} \frac{1-e^{-f(s,t)}}{f(s,t)} \frac{f(s,t)}{t-s} =q_i(s). $$
This proves the first part of (v). Next we prove the
second part.

To this end, first note that (\ref{s13}) implies that for any given
 $0<\varepsilon<1/3$ and $j\ne i$, there exists $0<\delta<1$ such that when $ 0<t-s<\delta$, we have
  $ P_{ii}(s,t)>1-\varepsilon,\  P_{jj}(s,t)>1-\varepsilon,\ {\rm and} \
  P_{ij}(s,t)<\varepsilon$. Since (\ref{s13}) holds uniformly in $j\in S$, for $ 0<h<t-s$, $A\subset S,\ {\rm and} \ i\not\in A$, it follows from (\ref{s002}) that
 \be
 \label{s006}
 \quad\quad\quad\quad\quad\quad  P_{i,A}(s,t) \ge (1-3\varepsilon) \sum_{l=1}^n
   P_{i,A}(s+(l-1)h,s+lh).
 \ee
In particular, taking  $A=S-\{i,j\}$ in (\ref{s006}), we obtain
   \be
   \label{s007}
  \quad\quad\quad\quad  P_{i, S-\{i,j\}}(s,t) \ge (1-3\varepsilon) \sum_{l=1}^n P_{i,
    S-\{i,j\}}(s+(l-1)h,s+lh).
   \ee
On the other hand, taking $B=A=S-\{i\} \ {\rm and} \  m=n$ in (\ref{s004}), it follows that
 \bee
   P_{i, S-\{i\}}(s,t) &=&  \sum_{l=1}^n \sum_{k\in S-\{i\}}
    {}_AP_{ik}(s,s+lh) P_{k, S-\{i\}}(s+lh,t)  \\
    &  &  +\; {}_AP_{ii}(s,s+nh) P_{i, S-\{i\}}(s+nh,t)\\
   & \le &  \sum_{l=1}^n  \sum_{k\in S-\{i\}}  {}_AP_{ik}(s,s+lh) + {}_AP_{ii}(s,s+nh) P_{i, S-\{i\}}(s+nh,t) \\
   & = &  \sum_{l=1}^n   {}_AP_{i, S-\{i\}}(s,s+lh) + {}_AP_{ii}(s,s+nh) P_{i, S-\{i\}}(s+nh,t) \\
     & = &  \sum_{l=1}^n  {}_AP_{ii}(s,s+(l-1)h) P_{i,
     S-\{i\}}(s+(l-1)h,s+lh) \\
     & &  +\; {}_AP_{ii}(s,s+nh) P_{i, S-\{i\}}(s+nh,t)\ \ \ ({\rm since}\ A=S-\{i\});
\eee
 consequently,
 \beb
  \quad\quad\quad\quad\quad   P_{i,S-\{i\}}(s,t) & \le &  \sum_{l=1}^n   P_{i, S-\{i\}}(s+(l-1)h,s+lh) \nonumber\\
\label{s008}
  & &  + \; {}_AP_{ii}(s,s+nh) P_{i, S-\{i\}}(s+nh,t).
\eeb
Subtracting (\ref{s007}) from (\ref{s008}), and using
(\ref{s007}) again, we obtain
 \beb
\!\!  P_{ij}(s,t) &\le&  \sum_{l=1}^n P_{ij}(s+(l-1)h,s+lh) +
 3\varepsilon \sum_{l=1}^n
  P_{i, S-\{i,j\}}(s+(l-1)h,s+lh) \nonumber\\
  & & + \; {}_AP_{ii}(s,s+nh) P_{i, S-\{i\}}(s+nh,t) \nonumber\\
   &\le&  \sum_{l=1}^n P_{ij}(s+(l-1)h,s+lh) +
   \frac{3\varepsilon}{1-3\varepsilon}  P_{i, S-\{i,j\}}(s,t) \nonumber\\
 \label{s009}
  & & + \; {}_AP_{ii}(s,s+nh) P_{i, S-\{i\}}(s+nh,t).
\eeb
 Recalling that $\varepsilon$ was arbitrary, taking the limit
of both sides of (\ref{s009}) as $\varepsilon \to 0$, we see that
  \be\label{s012}
  P_{ij}(s,t) \le  \sum_{l=1}^n P_{ij}(s+(l-1)h,s+lh)
   + \; {}_AP_{ii}(s,s+nh) P_{i, S-\{i\}}(s+nh,t).
\ee
 Summarizing, by (\ref{s002})
  \be
  \label{s010}
 \quad\quad\quad\quad\quad  (1-3\varepsilon) \sum_{l=1}^n  \frac{P_{ij}(s+(l-1)h,s+lh)}{nh} \le
  \frac{P_{ij}(s,t)}{nh},
  \ee
whereas by (\ref{s012})
 \be
  \label{s013}
    \frac{P_{ij}(s,t)}{nh}\le\sum_{l=1}^n  \frac{P_{ij}(s+(l-1)h,s+lh)}{nh}
    +\frac{ {}_AP_{ii}(s,s+nh) P_{i, S-\{i\}}(s+nh,t)}{nh}.
  \ee

To conclude, note that $nh\to t-s$ as $h\to 0^+$. Hence, using (\ref{s13}), (\ref{s010}), and (\ref{s013}) we obtain
 \be
  \label{s003}
 \quad\quad\quad\quad  \limsup_{h\to 0^+} \sum_{l=1}^n  \frac{P_{ij}(s+(l-1)h,s+lh)}{nh}
   = \limsup_{t-s\to 0^+}  \frac{P_{ij}(s,t)}{t-s}.
   \ee
From this equality and (\ref{s010}) it follows that
$$ \limsup_{t-s\to 0^+}  \frac{P_{ij}(s,t)}{t-s}
\le \frac{1}{1-3\varepsilon} \frac{P_{ij}(s,t)}{t-s}. $$
Hence, taking the limit infimum of both sides we obtain
$$ \limsup_{t-s\to 0^+}  \frac{P_{ij}(s,t)}{t-s}
      \le  \frac{1}{1-3\varepsilon} \liminf_{t-s\to 0^+} \frac{P_{ij}(s,t)}{t-s}.  $$
Finally, letting  $ \varepsilon\to 0$ we conclude the proof of part (v).

(vi) By (2.1), $P_{ii}(s,t)+\sum_{j\not=i} P_{ij}(s,t)\leq 1$ for
all $i,j\in S$ and $t\geq s\geq 0$, or, equivalently,
$\sum_{j\not=i}P_{ij}(s,t)\leq 1-P_{ii}(s,t)$. Hence (vi) follows
from (v) and Fatous' Lemma. $ \hfill$  $\Box$

\subsection{Proof of Theorem \ref{dd1}}
\label{sec2:subsec2}

\begin{proof}
(i) By Assumption A and the definition of $q_{ij}(t)$, the functions $\ol
P_{ij}^{(n)}(s,t)$ and $\ol Q_{ij}^{(n)}(s,t)$ are well defined for
every $n\geq 0$.

To prove (\ref{pqe}), we use induction on $n\ge 0$. Obviously,
 $\ol P_{ij}^{(0)}(s,t) = \ol Q_{ij}^{(0)}(s,t)$ and
 $\ol P_{ij}^{(1)}(s,t) = \ol Q_{ij}^{(1)}(s,t)$, by (\ref{s33}) and
 (\ref{q02}). Now assume that (\ref{pqe}) holds for some $n\ge 0$;
 we will show that it holds for $n+1$. By the induction hypothesis and
 (\ref{s33}) and (\ref{q02}), we have
 \bee
  \ol P_{ij}^{(n+1)}(s,t) &=& \int_s^t \sum\limits_{k\in S}\ e^{\int_s^u q_{ii}(v)dv} [q_{ik}(u)+d_{ik}(u)] \ol Q_{kj}^{(n)}(u,t) du
  \\
  &=& \int_s^t \sum\limits_{k\in S}\ e^{\int_s^u q_{ii}(v)dv} [q_{ik}(u)+d_{ik}(u)]\cdot
  \\
  & & \left[ \int_u^t \sum\limits_{l\in S}\ \ol Q_{kl}^{(n-1)}(u,x)\ e^{\int_x^t  q_{jj}(v)dv} [q_{lj}(x)+d_{lj}(x)] dx \right] du
  \\
  &=& \int_s^t \sum\limits_{l\in S}\ e^{\int_x^t  q_{jj}(v)dv} [q_{lj}(x)+d_{lj}(x)]\cdot
  \\
  & & \left[ \int_s^x \sum\limits_{k\in S}\ e^{\int_s^u  q_{ii}(v)dv} [q_{ik}(u)+d_{ik}(u)] \ol Q_{kl}^{(n-1)}(u,x) du \right] dx
  \\
  &=& \int_s^t \sum\limits_{l\in S}\ e^{\int_x^t  q_{jj}(v)dv}
  [q_{lj}(x)+d_{lj}(x)] \ol P_{il}^{(n)}(s,x) dx
  \\
  &=& \int_s^t \sum\limits_{l\in S}\ \ol Q_{il}^{(n)}(s,x)\  e^{\int_x^t  q_{jj}(v)dv}
  [q_{lj}(x)+d_{lj}(x)] dx
  \\
  &=& \ol Q_{ij}^{(n+1)}(s,t),
  \eee
Consequently, (\ref{pqe}) holds for $n+1$. This completes the
induction and verifies (i).

(ii) To prove (\ref{s35}), we first prove, by induction on $n$, the
following statement: for every $n\ge 0\ {\rm and} \ t\ge s\ge 0$, we
have
 \be
 \label{s36}
  \!\!\!\!\!\!\intsss \ol P_{ik}^{(n+1)}(s,v)d_{kj}(v) dv =\!\!\intsss \ol P_{ik}^{(n)}(s,v)
  [q_{kj}(v)+d_{kj}(v)]dv - \ol P_{ij}^{(n+1)}(s,t).
 \ee

Indeed, by (\ref{s32}) and (\ref{s31}), for $n=0$ a direct calculation gives
$$\intsss \ol P_{ik}^{(0)}(s,v)d_{kj}(v) dv = \delta_{ij}-\ol P_{ij}^{(0)}(s,t),$$
and so (\ref{s33}) together with (\ref{s32}) and (\ref{s31}) gives
 \bee
  \lefteqn{\intsss \ol P_{ik}^{(1)}(s,v)d_{kj}(v) dv} \\
  &=&  \intsss \int_s^v \sum\limits_{l\in S} \e{s}{u}{x} [q_{il}(u)+d_{il}(u)] \delta_{lk}\  e^{\int_u^v  q_{ll}(x)dx}\ \delta_{kj} (-q_{kk}(v)) du dv
  \\
  &=& \intsss \int_s^v  \e{s}{u}{x} [q_{ik}(u)+d_{ik}(u)] \delta_{kj}\ e^{\int_u^v  q_{kk}(x)dx} (-q_{kk}(v)) du dv
  \\
  &=& \intsss  \e{s}{u}{x} [q_{ik}(u)+d_{ik}(u)] \delta_{kj} \left[\int_u^t e^{\int_u^v  q_{kk}(x)dx} (-q_{kk}(v)) dv \right] du
  \\
  &=& \intsss  \e{s}{u}{x} [q_{ik}(u)+d_{ik}(u)] \delta_{kj} (1-e^{\int_u^t q_{kk}(x)dx}) du
  \\
  &=& \intsss \delta_{ik} \e{s}{u}{x} [q_{kj}(u)+d_{kj}(u)] du
  \\
  & & - \intsss \e{s}{u}{x} [q_{ik}(u)+d_{ik}(u)] \  \ol P_{kj}^{(0)}(u,t) du
  \\
  &=& \intsss  \ol P_{ik}^{(0)}(s,u)[q_{kj}(u)+d_{kj}(u)]du - \ol P_{ij}^{(1)}(s,t).
  \eee
Hence, (\ref{s36}) holds for $n=0$.

Now assume that (\ref{s36}) holds for some $n\geq 0$; we will show that it
holds for $n+1$. By the induction hypothesis and (\ref{s33}), we
 obtain
 \bee
  \lefteqn{ \intsss \ol P_{ik}^{(n+2)}(s,v)d_{kj}(v) dv} \\
  &=& \intsss \left[ \int_s^v \sum\limits_{l\in S} \e{s}{u}{x} [q_{il}(u)+d_{il}(u)] \ \ol P_{lk}^{(n+1)}(u,v) \ du \right] d_{kj}(v) dv
  \\
  &=& \intst \sss{l} \e{s}{u}{x} [q_{il}(u)+d_{il}(u)]\left[ \int_u^t \sss{k} \ol P_{lk}^{(n+1)}(u,v)d_{kj}(v) dv\right] du
  \\
  &=& \intst \sss{l} \e{s}{u}{x} [q_{il}(u)+d_{il}(u)]\cdot
  \\
  & &\left[\int_u^t \sss{k} \ol P_{lk}^{(n)}(u,v)[d_{kj}(v)+q_{kj}(v)] dv  -\ol P_{lj}^{(n+1)}(u,t)\right]du
  \\
  &=& \intst \sss{k} \int_u^t \sss{l} \e{s}{u}{x} [q_{il}(u)+d_{il}(u)]\ \ol P_{lk}^{(n)}(u,v)[d_{kj}(v)+q_{kj}(v)]\ dv du
   \\
  & & -\intst \sss{l} \e{s}{u}{x} [q_{il}(u)+d_{il}(u)]\ \ol P_{lj}^{(n+1)}(u,t)du
  \\
  &=& \intst \sss{k} \left[\int_s^v \sss{l} \e{s}{u}{x}[q_{il}(u)+d_{il}(u)]\ \ol P_{lk}^{(n)}(u,v) du \right]\cdot
  \\
  & & [d_{kj}(v)+q_{kj}(v)]dv - \ol P_{ij}^{(n+2)}(s,t) \\
  &=& \intsss \ol P_{ik}^{(n+1)}(s,v)[d_{kj}(v)+q_{kj}(v)]dv - \ol P_{ij}^{(n+2)}(s,t).
     \eee
Consequently, (\ref{s36}) holds for $n+1$ and this completes the induction.

Now note that (\ref{s36}) gives
 $$
   \sum_{n=0}^{\infty} \ol P_{ij}^{(n+1)}(s,t)  =\!\! \intsss \!
\sum_{n=0}^{\infty} \ol P_{ik}^{(n)}(s,v)[d_{kj}(v)+q_{kj}(v)]dv  -
\intsss \! \sum_{n=0}^{\infty} \ol
  P_{ik}^{(n+1)}(s,v)d_{kj}(v)dv.
 $$
This equality and (2.12) yield
 \bee
 \!\!\!\!\! \sum_{n=0}^{\infty} \ol P_{ij}^{(n)}(s,t)
  &=& \delta_{ij}-\intsss  \ol P_{ik}^{(0)}(s,v)d_{kj}(v) dv + \sum_{n=0}^{\infty} \ol
  P_{ij}^{(n+1)}(s,t)\\
  &=& \delta_{ij} + \intsss \ol{P}_{ik}(s,v)(d_{kj}(v)+q_{kj}(v))dv -\intsss
  \ol{P}_{ik}(s,v)d_{kj}(v) dv \\
  &=& \intsss \ol{P}_{ik}(s,v) q_{kj}(v)dv + \delta_{ij},
  \eee
 which proves (\ref{s35}).

(iii) The proof of (\ref{qbackw}) is quite similar to that of
(\ref{s35}). We first prove, by induction on $n$, the following
statement, which is analogous to (3.12): for every $n\ge 0\ {\rm
and} \ t\ge s\ge 0$,
 \be
 \label{q04}
  \!\!\!\!\!\intsss d_{ik}(v) \ol Q_{kj}^{(n+1)}(v,t) dv =\!\!\intsss [q_{ik}(v)+d_{ik}(v)] \ol
  Q_{kj}^{(n)}(v,t) dv - \ol Q_{ij}^{(n+1)}(s,t).
 \ee
Indeed, by (\ref{q01}), for $n=0$ we obtain
$$\intsss d_{ik}(v) \ol Q_{kj}^{(0)}(v,t) dv = \delta_{ij}-\ol Q_{ij}^{(0)}(s,t).$$
This fact and (\ref{q02}) yield
 \bee
 \quad \lefteqn{ \intsss d_{ik}(v) \ol Q_{kj}^{(1)}(v,t) dv} \\
  &=& \intsss d_{ik}(v) \left[ \int_v^t \sum\limits_{l\in S} \ol Q_{kl}^{(0)}(v,u) e^{\int_u^t q_{jj}(x)dx} [q_{lj}(u)+d_{lj}(u)]\ du \right] dv
  \\
  &=& \intst \int_v^t \sum\limits_{l\in S} \delta_{il} (-q_{ii}(v))\ e^{\int_v^u  q_{ii}(x)dx} e^{\int_u^t q_{jj}(x)dx} [q_{lj}(u)+d_{lj}(u)]\ du dv
  \\
  &=& \intst \sum\limits_{l\in S} \delta_{il} \left[\int_s^u (-q_{ii}(v))\ e^{\int_v^u  q_{ii}(x)dx} dv \right] e^{\int_u^t q_{jj}(x)dx} [q_{lj}(u)+d_{lj}(u)]\ du
  \\
  &=& \intst \sum\limits_{l\in S} \delta_{il} (1 - e^{\int_s^u  q_{ii}(x)dx})\ e^{\int_u^t q_{jj}(x)dx} [q_{lj}(u)+d_{lj}(u)]\ du
  \\
  &=& \intst \sum\limits_{l\in S} [q_{il}(u)+d_{il}(u)]\ \delta_{lj}\ e^{\int_u^t q_{ll}(x)dx}\ du
  \\
  & & - \intst \sum\limits_{l\in S} \ol Q_{il}^{(0)}(s,u) e^{\int_u^t q_{jj}(x)dx} [q_{lj}(u)+d_{lj}(u)]\ du
  \\
  &=& \intst \sum\limits_{l\in S} [q_{il}(u)+d_{il}(u)] \ol Q_{lj}^{(0)}(u,t) du - \ol Q_{ij}^{(1)}(s,t).
  \eee
Hence, (\ref{q04}) holds for $n=0$.

Now assume that (\ref{q04}) holds for some $n\ge 0$; we will show
that it holds for $n+1$. By the induction hypothesis and
(\ref{q02}), we obtain
 \bee
   \lefteqn{ \intsss d_{ik}(v) \ol Q_{kj}^{(n+2)}(v,t) dv }\\
  &=& \intsss d_{ik}(v) \left[ \int_v^t \sum\limits_{l\in S} \ol Q_{kl}^{(n+1)}(v,u) e^{\int_u^t q_{jj}(x)dx} [q_{lj}(u)+d_{lj}(u)]\ du \right] dv
  \\
  &=& \intst \sum\limits_{l\in S} \left[\int_s^u \sum\limits_{k\in S} d_{ik}(v) \ol Q_{kl}^{(n+1)}(v,u)\ dv \right] e^{\int_u^t q_{jj}(x)dx} [q_{lj}(u)+d_{lj}(u)]\ du
  \\
  &=& \intst \sum\limits_{l\in S} \left[\int_s^u \sum\limits_{k\in S} [q_{ik}(v)+d_{ik}(v)] \ol Q_{kl}^{(n)}(v,u)\ dv -  \ol Q_{il}^{(n+1)}(s,u) \right]\cdot
  \\[6pt]
  & & e^{\int_u^t q_{jj}(x)dx} [q_{lj}(u)+d_{lj}(u)]\ du
  \\[6pt]
  &=& \intst \sum\limits_{l\in S} \left[\int_s^u \sum\limits_{k\in S} [q_{ik}(v)+d_{ik}(v)] \ol Q_{kl}^{(n)}(v,u)\ dv \right]e^{\int_u^t q_{jj}(x)dx} [q_{lj}(u)+d_{lj}(u)]\ du
  \\
  & & - \intst \sum\limits_{l\in S} \ol Q_{il}^{(n+1)}(s,u)\ e^{\int_u^t q_{jj}(x)dx} [q_{lj}(u)+d_{lj}(u)]\ du
  \\
  &=& \intsss [q_{ik}(v)+d_{ik}(v)] \left[\int_v^t \sum\limits_{l\in S}  \ol Q_{kl}^{(n)}(v,u)\ e^{\int_u^t q_{jj}(x)dx} [q_{lj}(u)+d_{lj}(u)] du \right] dv
  \\[6pt]
  & & - \ol Q_{ij}^{(n+2)}(s,t)
  \\[6pt]
  &=& \intsss [q_{ik}(v)+d_{ik}(v)] \ol Q_{kj}^{(n+1)}(v,t)\ dv -  \ol Q_{ij}^{(n+2)}(s,t).
  \eee
Consequently, (\ref{q04}) holds for $n+1$ and this completes the
induction.

Now note that (\ref{q04}) gives
$$ \sum_{n=0}^{\infty} \ol
Q_{ij}^{(n+1)}(s,t) =\!\! \int_s^t \! \sum_{k\in S}
  [q_{ik}(v)+d_{ik}(v)] \sum_{n=0}^{\infty} \ol Q_{kj}^{(n)}(v,t)dv
   - \int_s^t \! \sum_{k\in S} d_{ik}(v) \sum_{n=0}^{\infty} \ol Q_{kj}^{(n+1)}(v,t)dv.
 $$
Therefore, by (\ref{q03}),
 \bee
 \!\!\!\!\!\!  \sum_{n=0}^{\infty} \ol Q_{ij}^{(n)}(s,t)
  &=& \delta_{ij}-\intsss  d_{ik}(v) \ol Q_{kj}^{(0)}(v,t) dv + \sum_{n=0}^{\infty} \ol Q_{ij}^{(n+1)}(s,t)\\
  &=& \delta_{ij} + \intsss [q_{ik}(v)+d_{ik}(v)] \ol Q_{kj}(v,t)dv
  -\intsss d_{ik}(v) \ol Q_{kj}(v,t) dv \\
  &=& \intsss q_{ik}(v) \ol Q_{kj}(v,t) dv + \delta_{ij}.
  \eee
This equality and (2.16) give (\ref{qbackw}).

(iv) To prove that $\ol P(s,t)$ is a nonhomogeneous
$Q(t)$-transition matrix we need to show that $\ol P_{ij}(s,t)$
 satisfies (\ref{s11})--(\ref{s13}), and that the partial
derivatives in (i)--(ii) of Definition 3 exist.

To prove (2.1) we already know that $\ol P_{ij}^{(0)}(s,t)\ge 0$, by (\ref{s32}).
Suppose now that $\ol P_{ij}^{(n)}(s,t)\ge 0$ for some $n$. To prove that this
holds for $n+1$, we use (\ref{s33}) and (\ref{s21}) to obtain
 \bee
\quad \ol P_{ij}^{(n+1)}(s,t) & =& \int_s^t \sum\limits_{k\in S}
\e{s}{u}{v} [q_{ik}(u)+d_{ik}(u)]  \ol P_{kj}^{(n)}(u,t) du
 \\
 &=&   \intst   \sum_{k\ne i \atop k\in S}  \e{s}{u}{v} q_{ik}(u) \ol P_{kj}^{(n)}(u,t)
 du \ge 0 \qquad \mbox{(by (\ref{s31}))}.
 \eee
Hence $ \ol P_{ij}(s,t)\ge 0$.

To prove that $ \ol P_{ij}(s,t)$ satisfies the second part of
(\ref{s11}), it suffices to show that
 \be
 \label{s37}
 \quad\quad\quad\quad\quad\quad\quad\quad\quad  \sss{j} \ol P_{ij}^{(n)}(s,t) \le 1\  \  \forall n\ge 0,
  \ee
because then in a similar manner we can prove show that
  \be
 \label{s38}
 \quad\quad\quad \quad\quad\quad \quad\quad \sss{j} \sum_{n=0}^N  \ol P_{ij}^{(n)}(s,t) \le 1 \ \ \ \forall N\ge 0.
 \ee
Therefore, since
$$  \ol P_{ij}(s,t)  = \sum_{n=0}^{\infty} \ol P_{ij}^{(n)}(s,t)=\lim_{N\to \infty} \sum_{n=0}^N  \ol P_{ij}^{(n)}(s,t),$$
(\ref{s11}) follows. Now, to prove (\ref{s37}), we use induction on
$n$. For $n=0$, (3.14) trivially holds, by (\ref{s32}).

Suppose now that (\ref{s37}) holds for some $n$, that is,
$$ \sss{j} \ol P_{ij}^{(n)}(s,t) \le 1.$$
To see that this holds for $n+1$, we use (\ref{s33}) and monotone convergence to obtain
 \bee
\! \sss{j} \ol P_{ij}^{(n+1)}(s,t)
 &=& \intsss \e{s}{u}{v} [q_{ik}(u)+d_{ik}(u)] \left[\sss{j} \ol P_{kj}^{(n)}(u,t)\right] du
   \\
   &\le & \intst \e{s}{u}{v} \sss{k} q_{ik}(u) du + \intsss \e{s}{u}{v} d_{ik}(u)du.
   \eee
Hence, by (\ref{s22}) and (\ref{s31}),
$$\sss{j} \ol P_{ij}^{(n+1)}(s,t) \le \intsss \e{s}{u}{v} \delta_{ik}(- q_{ii}(u)) du
= 1-\e{s}{t}{u} \le  1,$$
which yields (\ref{s37}).

Now we will verify that $  \ol P_{ij}(s,t)$ satisfies the C-K
equation (\ref{s12}). Observe that this holds if and only if, for
every $n\geq 0$ and $s\le u \le t$,
 \be
  \label{s39}
 \quad\quad\quad\quad\quad\quad\quad  \ol P_{ij}^{(n)}(s,t) =  \sum_{m=0}^n   \sss{k}    \ol P_{ik}^{(m)}(s,u)  \ol P_{kj}^{(n-m)}(u,t).
   \ee

We will prove (3.16) by induction. In fact, for $n=0$
(\ref{s39}) follows from (\ref{s32}). Suppose now that (\ref{s39})
holds for some $n\geq 0$. To prove (\ref{s39}) for $n+1$, we use the
induction hypothesis and (\ref{s33}), to obtain, for any $s\le r\le t$,
 \bee
\!\!\!\!\!\!  \ol P_{ij}^{(n+1)}(s,t)
  &=& \intsss \e{s}{u}{v} [q_{ik}(u)+d_{ik}(u)] \ol P_{kj}^{(n)}(u,t) du
  \\
  &=& \int_s^r \sss{k} \e{s}{u}{v} [q_{ik}(u)+d_{ik}(u)] \left[\sum_{m=0}^n  \sss{l}   \ol P_{kl}^{(m)}(u,r) \ol P_{lj}^{(n-m)}(r,t) \right] du
  \\
  & & +\int_r^t \sss{k}  \e{s}{u}{v} [q_{ik}(u)+d_{ik}(u)]  \ol P_{kj}^{(n)}(u,t) du
  \\
  &=& \sum_{m=0}^n \sss{l}  \ol P_{il}^{(m+1)}(s,r) \ol P_{lj}^{(n-m)}(r,t)+B,
  \eee
where$$B =\int_r^t \sss{k}  \e{s}{u}{v} [q_{ik}(u)+d_{ik}(u)]
 \ol P_{kj}^{(n)}(u,t) du. $$
Since (\ref{s39}) holds for $n=0$, recalling (2.10) we then have
  $$B= \ol P_{ii}^{(0)}(s,r)  \ol P_{ij}^{(n+1)}(r,t)= \sss{l}  \ol P_{il}^{(0)}(s,r) \ol P_{lj}^{(n+1)}(r,t), $$
and so
$$  \ol P_{ij}^{(n+1)}(s,t)= \sum_{m=0}^{n+1} \sss{l}  \ol P_{il}^{(m)}(s,r)
   \ol P_{lj}^{(n+1-m)}(r,t). $$
Hence, (\ref{s39}) holds for all $n$, and, as already noted,
 (\ref{s12}) follows for $\ol P_{ij}(s,t)$.

To see that $\ol P_{ij}(s,t)$ satisfies (\ref{s13}), using
(\ref{s31})-(\ref{s34}), (\ref{s21}), (\ref{s22}) and (\ref{s38}),
we obtain
\bee
 \lefteqn{ |\ \ol P_{ij}(s,t) - \delta_{ij}| }\\
  &\le& (1-\e{s}{t}{v}) + \sum_{n=0}^\infty \ol P_{ij}^{(n+1)}(s,t)
  \\
  &=& (1-\e{s}{t}{v}) + \intsss \e{s}{u}{v} [q_{ik}(u)+d_{ik}(u)] \sum_{n=0}^\infty \ol P_{kj}^{(n)}(u,t) du
  \\
  &\le& (1-\e{s}{t}{v}) + \intsss \e{s}{u}{v} [q_{ik}(u)+d_{ik}(u)] du
  \\
  &\le& (1-\e{s}{t}{v}) + \intst \e{s}{u}{v} [-q_{ii}(u)] du
  \\
  &\to& 0\ \ \ {\rm as}\  t\to s^{+}.
  \eee
This implies the desired result.

To summarize, we have just shown that $ \ol P(s,t)
=(\ol P_{ij}(s,t),\ \ i,j\in S)$ is a nonhomogeneous pretransition
 matrix. Therefore, to complete the proof that it is a nonhomogeneous $Q(t)$-transition matrix,
it only remains to verify that $\ol P(s,t)$ satisfies (i) and (ii) in Definition 3. But in fact from (\ref{s35}), (\ref{qbackw}) and the
fundamental theorem of calculus for Lebesgue integrals we can obtain a bit more than (i),
namely, {\it Kolmogorov's forward}  and {\it backward equations}
  \begin{eqnarray}
   &&\quad\quad\quad\quad\quad\quad\quad\quad\  \frac{\partial \ol P_{ij}(s,t)}{\partial t}=\sss{k}\ol P_{ik}(s,t) q_{kj}(t), \nonumber  \\
   &&\quad\quad\quad\quad\quad\quad\quad\quad  \frac{\partial \ol P_{ij}(s,t)}{\partial s} = - \sss{k}
    q_{ik}(s) \ol P_{kj}(s,t)  \nonumber
 \end{eqnarray}
for a.a. $t\ge s\ge 0$. Moreover, if we take $t=s$ in the forward
equation, we obtain (2.8). This verifies (iv) and it also completes
the proof of Theorem 2.
\end{proof}


\subsection{Proof of Theorem \ref{dd3}}
\label{sec2:subsec4}

To prove Theorem \ref{dd3} we need the following facts.

\begin{lemma}
 \label{dd2}
 Assume that $Q(t)$ is a nonhomogeneous $Q(t)$-matrix satisfying
Assumption A. Then for any nonhomogeneous $Q(t)$-transition matrix
$P(s,t)=(P_{ij}(s,t),\ \ i,j\in S)$ we have, for every $i,j\in S$
and $0\le s< t<\infty$,
 \begin{enumerate}
 \item \be
  \label{s43}
  \!\!\!\!\!\!\!\!\!\!\!\!\!\!   \frac{\partial P_{ij}(s,t)}{\partial s} \le  - \sss{k} q_{ik}(s) P_{kj}(s,t);
\ee
  \item \be
  \label{s44}
 \!\!\!\!\!\!\!\!\!\!\! \!\!\!  P_{ij}(s,t) \ge  \intsss \e{s}{u}{v}
[q_{ik}(u)+d_{ik}(u)] P_{kj}(u,t) du
    + \delta_{ij} \e{s}{t}{u}\!\!\!.
\ee
\end{enumerate}
\end{lemma}

\begin{proof}
 (i) To prove (\ref{s43}),   we use the C-K equation (\ref{s12})  to obtain
 \beb
  \!\!\!\!\!\!\!&&\!\!\!\!\! \frac {1}{h} [P_{ij}(s+h,t)-P_{ij}(s,t)] \nonumber  \\
  \label{s47}
  \!\!\!\!\!\!\!&=&\!\!\!\!\! \frac {1}{h} [1- P_{ii}(s,s+h)] P_{ij}(s+h,t)\!-\!\!\!\sum_{k\ne i \atop k\in S} \frac {1}{h}
  P_{ik}(s,s+h)P_{kj}(s+h,t).
  \eeb
Hence, by Fatou's Lemma, (\ref{s23}) and Theorem \ref{dd0}(iii) we have
 \be
  \label{s48}
 \quad\quad\quad\quad  \liminf_{h\to 0^+}\sum_{k\ne i \atop k\in S} \frac {1}{h}
  P_{ik}(s,s+h)P_{kj}(s+h,t) \ge \sum_{k\ne i \atop k\in S} q_{ik}(s)
  P_{kj}(s,t).
  \ee
Then (i) follows from Definition 3, (\ref{s47}), and (\ref{s48}).

(ii) By (\ref{s43}), we have
 \bee
    \lefteqn{ \intsss \e{s}{u}{v} [q_{ik}(u)+d_{ik}(u)] P_{kj}(u,t) du}
    \\
    &=& \intsss \e{s}{u}{v} q_{ik}(u)P_{kj}(u,t) du +
      \intsss \e{s}{u}{v} d_{ik}(u)P_{kj}(u,t) du   \\
    & \le & -  \intst \e{s}{u}{v} \frac{\partial P_{ij}(u,t)}{\partial u}
    du + \intst \e{s}{u}{v} (-q_{ii}(u)) P_{ij}(u,t) du \\
    &=&  -\delta_{ij} \e{s}{t}{v}  + P_{ij}(s,t).
  \eee
This yields (\ref{s44}), and completes the proof of Lemma 2.
\end{proof}

With Lemma \ref{dd2}, we can easily prove Theorem \ref{dd3}.

\noindent
{\it Proof of Theorem \ref{dd3}.}

\vspace{2mm}
 (i)\ Let $ \ol P_{ij}^{(n)}(s,t) \ {\rm and} \ \ol P_{ij}(s,t)$ be as defined
 in Theorem \ref{dd1}, and let $P_{ij}(s,t)$ be as in Lemma \ref{dd2}. We will prove (2.19) by showing that (3.21), below, holds for all $n\ge 0$.

By Lemma \ref{dd2}, we know that $P_{ij}(s,t)$ satisfies (\ref{s43}), which for $i=j$ becomes
$$
  \frac{\partial P_{ii}(u,t)}{\partial u} \le  - \sss{k} q_{ik}(u) P_{ki}(u,t) \ \ \forall \ i\in S,\ \ 0\le u\le
  t<\infty. $$
  Since $q_{ij}(u)\ge 0 \ {\rm for} \ i\ne j$, we have
$ \frac{\partial  P_{ii}(u,t)}{\partial u} \le  -  q_{ii}(u)
P_{ii}(u,t) $, and so $  -\frac{d P_{ii}(u,t)}{P_{ii}(u,t)} \ge
q_{ii}(u) du $. Integrating from $s$ to $t$  on both sides and using
(\ref{s32}), we obtain $  P_{ii}(s,t) \ge  \ol P_{ii}^{(0)}(s,t)$.
Hence
$$  P_{ij}(s,t) \ge  \ol P_{ij}^{(0)}(s,t) \ \ \forall \ i,j\in S. $$

Suppose now that for some $n\ge 0$,
 \be
   \label{s56}
  \quad\quad\quad\quad\quad\quad\quad\quad\quad\quad\quad
   P_{ij}(s,t) \ge  \sum_{m=0}^n  \ol P_{ij}^{(m)}(s,t).
\ee To prove that this holds for $n+1$, we use (\ref{s44}) and the
induction hypothesis to obtain
  \bee
 \!\!  P_{ij}(s,t)
  & \ge & \delta_{ij} \e{s}{t}{u} + \intsss \e{s}{u}{v}
  [q_{ik}(u)+d_{ik}(u)]
    P_{kj}(u,t) du    \\
  &\ge &  \delta_{ij} \e{s}{t}{u} + \sum_{m=0}^n  \intsss \e{s}{u}{v}
  [q_{ik}(u)+d_{ik}(u)]
     \ol P_{kj}^{(m)}(u,t) du  \\
  &=&   \ol P_{ij}^{(0)}(s,t)+ \sum_{m=0}^n  \ol P_{ij}^{(m+1)}(s,t)\quad\mbox{(by (2.10), (2.11))} \\
  &=&  \sum_{m=0}^{n+1}  \ol P_{ij}^{(m)}(s,t).
  \eee
Therefore, (\ref{s56}) holds for all $n\ge 0$, and letting $n\to \infty $ in (\ref{s56}), we obtain (\ref{s55}).

(ii)\ By Definition 3 and (\ref{s35}), we know that $ \ol P(s,t) =(\ol P_{ij}(s,t),\ \ i,j\in S )$  is regular if and only if
  $$  \sss{j}  \left[\int_s^t \sum\limits_{k\in
S} \ol P_{ik}(s,v)q_{kj}(v) dv +
     \delta_{ij} \right]  \equiv  1, $$
that is
 $$  \sss{j} \left[\int_s^t \sum\limits_{k\in S} \ol P_{ik}(s,v)q_{kj}(v) dv \right] \equiv 0, $$
which is the same as (2.20). This completes the proof of Theorem 3.
$ \hfill$  $\Box$

\section{ CONCLUSIONS}

In this paper we have presented a fairly detailed, self--contained,
exposition of the construction of a $Q(t)$-transition matrix
starting from a nonhomogeneous $Q(t)$-matrix that satisfies a very
mild measurability condition. Moreover, such a transition matrix is in fact
the {\it minimum} $Q(t)$-transition matrix and we have presented a
necessary and sufficient condition for it to be unique and regular.
In short, this paper efficiently generalizes the main results of a
nonhomogeneous $Q(t)$-transition matrix with continuous and
conservative transition rates $q_{ij}(t)$, to the case in which the
$q_{ij}(t)$ are measurable and may not be conservative.


\begin{thebibliography}{10}

\bibitem{Anderson91} Anderson, W.J. (1991). {\it Continuous Time Markov
Chains}. Springer, New York.

\bibitem{Ash00} Ash, R.B. (2000). {\it Probability and Measure
Theory, Second Edition}. Academic, London.

\bibitem {B03} Breuer, B. (2003). {\it From Markov Jump Processes to
Spatial Queues.} Kluwer, Dordrecht.

\bibitem{Chang06} Chang, H.S. (2006). Perfect information two-person zero-sum markov games
with imprecise transition probabilities. {\it  Math. Meth. Oper.
Res.} {\bf 64}, 335-351.

\bibitem {F40} Feller, W. (1940).  On the integro-differential
equations of purely discontinuous Markoff processes. {\it Trans.
Amer. Math. Soc.} {\bf 48}, 488-515.

\bibitem {GS96} Gikhman, I.I. and Skorokhod, A.V. (1996).  {\it Introduction
to the Theory of Random Processes}. Dover, Mineola, NY. (This is a
reprint of the 1969 edition published by Saunders, Philadelphia,
PA.).

\bibitem  {GH05} Guo, X.P. and Hern$\acute{\rm {a}}$ndez-Lerma, O. (2005).  Nonzero-sum games  for continuous--time  Markov chains with unbounded
discounted payoffs. {\it J. Appl. Probab.} {\bf 42}, 303--320.

\bibitem  {GH03a} Guo, X.P.  and Hern$\acute{\rm {a}}$ndez-Lerma,
O. (2003). Zero-sum games  for continuous--time  Markov chains with
unbounded transition and average payoff rates. {\it J. Appl.
Probab.} {\bf 40}, 327-345.

\bibitem  {GH03b} Guo, X.P.  and Hern$\acute{\rm {a}}$ndez-Lerma, O. (2003). Drift and monotonicity conditions for continuous--time controlled Markov chains with an average criterion. {\it IEEE Trans. Autom. Control} {\bf 48},
236-245.

\bibitem{Guo03} Guo, X.P. and Hern\'andez--Lerma,O. (2003).
 Continuous--time controlled Markov chains. {\it Ann. Appl. Prob.} {\bf 13},
 363--388.

\bibitem{Guo02} Guo, X.P. and Zhu, W.P. (2002).  Denumerable-state
continuous-time Markov decision processes with unbounded transition
and reward rates under the discounted criterion. {\it J. Appl.
Probab.} {\bf 39}, 233--250.

\bibitem{HuDH83} Hu, D.H. (1983).  {\it Markov Process Theory with Countable State}. Wuhan University Press, Wuhan. (In
Chinese.)

\bibitem{HuQY96} Hu, Q.Y. (1996).  Nonstationary
continuous time Markov decision processes with the expected total
reward criterion. {\it Optimization} {\bf 36}, 181--189.

\bibitem{K.P75} Kakumanu P. (1975).  Continuous time Markovian decision process with average return criterion. {\it J. Math. Anal. Appl.} {\bf 52},
173--188.

\bibitem{LiuJY04} Liu, J.Y., Hu, Q.Y. and Wang, J.M. (2004).  The foundational
 assumption of continuous--time Markov decision process. {\it Acta Math. Appl.
 Sinica} {\bf 27}, 756--759. (In Chinese.)

\bibitem{Puterman94} Puterman, M.L. (1994).  {\it Markov Decision
Processes}. Wiley, New York.

\bibitem{Sennott99} Sennott, L.I. (1999).  {\it  Stochastic Dynamic Programming and the Control of Queueing Systems}. Wiley, New
York.

\bibitem{Sinha04} Sinha, S. and Jana, S. (2004).  Semi-infinite semi-Markov stochastic games. {\it Opsearch} {\bf 41}, 278-290.

\bibitem{Stidham93} Stidham, S. and Weber, R. (1993).  A survey of
Markov decision models for the control of networks of queues. {\it
Queueing Systems} {\bf 13}, 291--314.

\bibitem{WuCB97a} Wu, C.B. (1997).  Continuous time Markov decision process with unbounded reward and nonuniformly bounded transtion rates under discounted criterion.
{\it Acta Math. Appl. Sinica} {\bf 20}, 196--208. (In Chinese.)





\end{thebibliography}
\end{document}